\documentclass{amsart}
\usepackage{ifthen}
\usepackage{amssymb,amsmath}
\usepackage{enumerate}
\usepackage[colorlinks=true]{hyperref}

\numberwithin{equation}{section}

\def\phi{\varphi}
\def\R{\mathbb R}
\def\N{\mathbb N}
\newcommand{\ch}{\mathcal{H}}
\newcommand{\cm}{\mathcal{M}}
\newcommand{\res}{\ensuremath{\,\textsf{\small \upshape L}\,}}

\newtheorem{theorem}{Theorem}[section]
\newtheorem{lemma}[theorem]{Lemma}
\newtheorem{corollary}[theorem]{Corollary}

\theoremstyle{definition}

\theoremstyle{remark}

\def\fuhome{@math.fu-berlin.de}

\begin{document}

\title{Uniqueness of compact tangent flows in Mean Curvature Flow}

\author{Felix Schulze}
\thanks{Partially supported by a Feodor-Lynen fellowship of the Alexander von
    Humboldt Foundation.}
\address{Felix Schulze: 
  Freie Universit\"at Berlin, Arnimallee 3, 
  14195 Berlin, Germany}
\curraddr{}
\email{Felix.Schulze\fuhome}


\dedicatory{}


\begin{abstract}
  We show, for mean curvature flows in Euclidean space, that if one of 
  the tangent flows at a given space-time point consists of a closed, 
  multiplicity-one, smoothly embedded  self-similar shrinker,
  then it is the unique tangent flow at that point.
  That is the limit of the parabolic rescalings does not depend on the 
  chosen sequence of rescalings. Furthermore, given such a
  closed, multiplicity-one, smoothly embedded self-similar shrinker
  $\Sigma$, we show that any solution of the rescaled flow, which is 
  sufficiently close to $\Sigma$, with Gaussian density ratios greater
  or equal to that of $\Sigma$, stays for all time close
  to $\Sigma$ and converges to a possibly different self-similarly
  shrinking solution $\Sigma'$. The central point in the argument is a 
  direct application of the Simon-{\L}ojasiewicz inequality to 
  Huisken's monotone Gaussian integral for Mean Curvature Flow.
\end{abstract}

\maketitle
\section{Introduction}
In this paper we study Mean Curvature Flow (MCF) of $n$-surfaces of
codimension $k \geq 1$ in
$\R^{n+k}$, which are close to self-similarly shrinking
solutions. In the smooth case we consider a family of embeddings 
$F:M^n\times (t_1,t_2) \rightarrow \R^{n+k}$, for $M^n$ closed, such that
\begin{equation*}
  \frac{d}{dt}F(p,t) = \vec{H}(p,t)\ ,
\end{equation*}
where $\vec{H}(p,t)$ is the mean curvature vector of $M_t:=F(M,t)$ at
$F(p,t)$. We
denote with $\cm = \bigcup_{t\in (t_1,t_2)}( M_t\times
\{t\})\subset \R^{n+k}\times \R$ its space-time track. 

In the following, let $\Sigma^n$ be a smooth, closed, embedded $n$-surface
in $\R^{n+k}$ where the mean curvature vector satisfies 
$$\vec{H}=-\frac{x^\perp}{2}.$$
Here $x$ is the position vector at a point on $\Sigma$ and ${}^\perp$
the projection to the normal space of $\Sigma$ at that point.
Such a surface gives rise to a self-similarly shrinking solution $\cm_\Sigma$,
where the evolving surfaces are given by
$$\Sigma_t= \sqrt{-t}\cdot\Sigma,\ \ \ t\in (-\infty,0).$$
We denote its space-time track by $\cm_\Sigma$. 

We also want to study the case that the flow is allowed
to be non-smooth. Following
\cite{Ilmanen}, we say that a family of Radon measures $(\mu_t)_{t\in[t_1,t_2)}$
on $\R^{n+k}$ is an integral $n$-Brakke flow, if for almost every $t$ 
the measure $\mu_t$ comes from a $n$-rectifiable varifold with integer 
densities. Furthermore, we require that 
given any $\phi \in
C_c^2(\R^{n+k};\R^+)$ the following inequality holds for every $t>0$
\begin{equation}
  \label{eq:intro2}
  \bar{D}_t\mu_t(\phi)\leq \int -\phi |\vec{H}|^2 + \langle\nabla \phi,
  \vec{H}\rangle \, d\mu_t,
\end{equation}
where $\bar{D}_t$ denotes the upper derivative at time $t$ and we take
the left hand side to be $-\infty$, if $\mu_t$ is not $n$-rectifiable, 
or does not carry a weak mean curvature. 
Note that if $M_t$ is moving smoothly by mean curvature flow, then 
$\bar{D}_t$ is just the usual derivative and we have
equality in \eqref{eq:intro2}.

We restrict to integral $n$-Brakke flows which are
close to a smooth self-similarly shrinking solution. 
The assumption that the Brakke flow is close in measure 
to a smooth solution with multiplicity one
actually yields that the Brakke flow has 
unit density. This implies that for almost all $t$ the
corresponding Radon measures can be written as  
$$ \mu_t = \ch^n\res M_t\, .$$
Here $M_t$ is a $n$-rectifiable subset of $\R^{n+k}$ and $\ch^n$ is
the $n$-dimensional Hausdorff-measure on $\R^{n+k}$.
If the flow is (locally) smooth, then $M_t$ can be (locally) represented by
a smooth $n$-surface evolving by MCF. Conversely, if
$M_t$ moves smoothly by MCF, then $\mu_t:=\ch^n\res M_t$ defines a unit density
$n$-Brakke flow.

\begin{theorem}
  \label{thm:mainthm}
  Let $\cm = (\mu_t)_{t\in (t_1,0)}$ with $t_1< 0$
  be an integral n-Brakke flow such that
  \begin{itemize}
  \item[i)] $(\mu_t)_{t\in (t_1,t_2)}$ is sufficiently close in measure
    to $\cm_\Sigma$ for some $t_1<t_2<0$.\\[-1ex]
  \item[ii)] $\Theta_{(0,0)}(\cm) \geq \Theta_{(0,0)}(\cm_\Sigma)$,
    where $\Theta_{(0,0)}(\cdot)$ is the respective Gaussian density
    at the point $(0,0)$ in space-time.
  \end{itemize}
  Then $\cm$ is a smooth flow for $t\in [(t_1+t_2)/2,0)$, and the
  rescaled surfaces $\tilde{M}_t:=(-t)^{-1/2}\cdot M_t$ can be written as normal graphs over
  $\Sigma$, given by smooth sections $v(t)$ of the normal bundle
  $T^\perp\Sigma$, with  $|v(t)|_{C^m(T^\perp\Sigma)}$ uniformly bounded 
  for all $t \in  [(t_1+t_2)/2,0)$ and all $m \in \N$. 
  Furthermore, there exists a self-similarly shrinking surface $\Sigma'$ with
  $$\Sigma' = \text{graph}_\Sigma(v')$$
  and 
  $$ |v(t)-v'|_{C^m} \leq c_m (\log(-1/t))^{-\alpha_m}$$
  for some constants $c_m>0$ and exponents $\alpha_m>0$ for all $m\in\N$. 
\end{theorem}

For the definition of the Gaussian density ratios and Gaussian density
we refer the reader to section \ref{sec:graphrep}.
The above theorem implies uniqueness of compact tangent flows as
follows. Let the
parabolic rescaling with a factor $\lambda>0$ be given by
$$ \mathcal{D}_\lambda: \R^{n+k}\times \R\rightarrow \R^{n+k}\times\R,
(x,t)\mapsto (\lambda x, \lambda^2t)\ .$$
Note that any Brakke flow $\cm$ (smooth MCF) is mapped to a
Brakke flow (smooth MCF), i.e. $\mathcal{D}_\lambda(\cm)$ is
again a Brakke flow (smooth MCF).

Let $(x_0,t_0)$ be a point in space-time and 
$(\lambda_i)_{i\in\N}, \lambda_i \rightarrow \infty,$ be a sequence of
positive numbers. If $\cm$ is a Brakke flow with bounded
area ratios, then the compactness theorem for Brakke flows (see
\cite[7.1]{Ilmanen}) ensures
that
\begin{equation}
  \label{eq:intro3}
  \mathcal{D}_{\lambda_i}(\cm - (x_0,t_0)) \rightarrow \cm'\ ,
\end{equation}
where $\cm'$ is again a Brakke flow. Such a flow is called a {\it tangent
  flow} of $\cm$ at $(x_0,t_0)$. Huisken's monotonicity formula
ensures that $\cm'$ is self-similarly shrinking, i.e. it is invariant
under parabolic rescaling.

\begin{corollary}
  \label{thm:maincor}
Let $\cm$ be an integral $n$-Brakke flow with bounded area ratios, 
and assume that at $(x_0,t)\in
\R^{m+k}\times \R$ a tangent flow of $\cm$ is $\cm_\Sigma$. Then this
tangent flow is unique, i.e. for any sequence $(\lambda_i)_{i\in\N}$
of positive numbers, $\lambda_i \rightarrow \infty$ it holds
$$ \mathcal{D}_{\lambda_i}(\cm - (x_0,t_0))\rightarrow \cm_\Sigma\, .$$
\end{corollary}
\medskip
Other than the shrinking sphere and the Angenent torus
\cite{Angenent92} no further examples of compact self-similarly
shrinking solutions in codimension one are known so far. However, several
numerical solutions of D.\,Chopp \cite{Chopp94} suggest that there are a 
whole variety of such solutions. In higher codimensions this class of
solutions should be even bigger. 

In a recent work
of Kapouleas/Kleene/M{\o}ller \cite{KapouleasKleeneMoeller11} and
X.H.\,Nguyen \cite{XuanHienNguyen11} non-trivial, non-compact,
self-similarly shrinking solutions were constructed. In
\cite{Huisken90}, G.\,Huisken showed that, under the assumption that the
second fundamental form is bounded, the only solutions in the
mean convex case are shrinking spheres and cylinders. The assumption
 on the second fundamental form was recently removed by
T.H.\,Colding and W.P.\,Minicozzi in \cite{ColdingMinicozzi09b}. They also
proved a smooth compactness theorem for closed self-similarly shrinking surfaces
of fixed genus in $\R^3$, see \cite{ColdingMinicozzi09a}.

The analogous problem for minimal surfaces is the uniqueness of tangent cones.
This was studied in \cite{Taylor73, AllardAlmgren81, White83}, and, in
the case of multiplicity one tangent cones with isolated
singularities, completely settled by L. Simon in \cite{Simon83}. 
One of the main tools in the analysis therein is the generalisation of 
an inequality due to {\L}ojasiewicz for real analytic functions to the
infinite dimensional setting.

Also in the present argument, this Simon-{\L}ojasiewicz inequality for
``convex'' energy functionals on closed surfaces, plays a central
role. We adapt several ideas from \cite{Simon83, LSimon96}. In section 
\ref{sec:graphrep} we recall Huisken's monotonicity formula and show
that any integral n-Brakke flow, which is close in measure 
to a smooth mean curvature flow, has unit density and is
smooth. Furthermore we prove
a smooth extension lemma for Brakke flows close to
$\cm_\Sigma$. We also introduce the rescaled flow. In section 
\ref{sec:convergence} we treat the Gaussian integral of Huisken's
monotonicity formula for the rescaled flow as an appropriate ``energy
functional'' on $\Sigma$ and use the Simon-{\L}ojasiewicz inequality to
prove a closeness lemma. This lemma and the extension lemma
are then applied to prove the main theorem and its corollary.
 
The author would like to thank K.\,Ecker, L.\,Simon, N.\,Wickramasekera,
J.\,Bern\-stein and B.\,White for discussions and helpful comments on 
this paper. He would also like to thank the Department of Mathematics
at Stanford University for their hospitality during the winter/spring
of 2008 while the main part of this work was completed.
\bigskip
\bigskip
\section{Graphical representation and rescaling}
\label{sec:graphrep}
As in the introduction, let $\cm$ be a smooth mean
curvature flow of embedded $n$-dimensional surfaces in
$\R^{n+k}$. Let
$$\rho_{x_0,t_0}(x,t) = \frac{1}{(4\pi(t_0-t))^{n/2}}\cdot \exp
\Big(-\frac{|x-x_0|^2}{4(t_0-t)}\Big),\ \ t<t_0$$
be the backward heat kernel centered at $(x_0,t_0)$. Huisken's
monotonicity formula states that for $t<t_0$
\begin{equation}
  \label{eq:huiskenmon}
  \frac{d}{dt}\int_{M_t} \rho_{x_0,t_0}(x,t)\, d\mathcal{H}^n = -
  \int_{M_t} \Big|\vec{H}+ \frac{x^\perp}{2(t_0-t)}\Big|^2\rho_{x_0,t_0}(x,t)\, d\ch^n\ .
\end{equation}
Thus the Gaussian density ratio 
$$ \Theta_{x_0,t_0}(\cm, t):=\int_{M_t} \rho_{x_0,t_0}(x,t)\,
d\mathcal{H}^n $$
is a decreasing function in $t$. If $t_0\in (t_1,t_2]$ then the limit
$$ \Theta_{x_0,t_0}(\cm)= \lim_{t\nearrow t_0} \Theta_{x_0,t_0}(t)$$
exists for all $x_0 \in \R^{n+k}$ and is called the Gaussian density
at $(x_0,t_0)$. The corresponding result also holds for Brakke flows,
see \cite[Lemma 7]{Ilmanen:SingMCF}. In that case the equality in
\eqref{eq:huiskenmon} is replaced by the appropriate inequality. It
can be deduced from \eqref{eq:huiskenmon} that any limit of parabolic
rescalings as in \eqref{eq:intro3} is self-similarly shrinking. This
is also true if the convergence and limit are not smooth, see 
\cite[Lemma 8]{Ilmanen:SingMCF}. Self-similarly shrinking solutions
$\cm_\Sigma$ can be characterized by the fact that the 
Gaussian density ratio $\Theta_{(0,0)}(\cm_\Sigma, t)$ is 
constant in time.  

For integral Brakke flows the monotonicity of the Gaussian density
ratios together with Brakke's local regularity theorem can 
be used to show that the flow is smooth with unit density, provided
that it is weakly close to a smooth mean curvature flow.

\begin{lemma}\label{lem:locreg} Let $\cm = (\mu_t)_{t\in (t_1,t_2]}$
  be an integral $n$-Brakke flow with bounded area ratios, and 
  let $\cm' =(M^n_t)_{t\in [t_1,t_2]}$ be a
  smooth, unit density mean curvature flow of embbeded surfaces. 
  Furthermore let $\Omega
  \Subset \bar{\Omega}$ be open subsets of $\R^{n+k}$. If 
$\cm$ is sufficiently close in measure to $\cm'$ on $\bar{\Omega}\times
(t_1,t_2]$, then $\cm$ is smooth on $\Omega\times ((t_1+t_2)/2,t_2]$
with bounds on all derivatives.
\end{lemma}

\begin{proof}
Choose an open subset $\tilde{\Omega} \subset \mathbb{R}^n$ such that
$\Omega\Subset \tilde{\Omega}\Subset \bar{\Omega}$. The flow $(M^n_t)_{t\in
  [t_1,t_2]}$ is smooth, so the second fundamental form of the surfaces
$M_t$ is uniformly bounded on $\tilde{\Omega}\times [t_1, t_2]$. Since the
surfaces $M_t$ are embedded, the closeness in measure of $\cm$ to
$\cm'$, together with
monotonicity of the Gaussian density ratios, implies that there exists
a $\delta>0$ such that for all $t\in [\tfrac{2}{3}t_1+\tfrac{1}{3}t_2,
t_2],\ x \in \text{spt}\mu_t$,
$$\Theta_{x,\tau}(\cm, t)\leq \frac{5}{4}$$
for all $\tau \in [t,t+\delta]$. The fact that the Radon measures
$\mu_t$ come for a.e. $\! t$ from an integer $n$-rectifiable varifold 
implies that for a.e. $t\in [\tfrac{2}{3}t_1+\tfrac{1}{3}t_2,
t_2]$ at a.e. $\! x \in \text{spt}\mu_t$
we have
$$\frac{\mu_t(B_\rho(x))}{\omega_n \rho^n}\leq \frac{3}{2} $$
for all $\rho$ sufficiently small. Thus $\cm$ has unit density on $
[\tfrac{2}{3}t_1+\tfrac{1}{3}t_2,t_2]$ and Brakke's local regularity theorem
, see  \cite[12.1]{Ilmanen}, implies that if $\cm$ is close in measure
to $\cm'$, then $\cm$ is smooth on
$\Omega\times ((t_1+t_2)/2,t_2]$ with bounds on all derivatives.
\end{proof}

Now let us assume that $\cm$ is smooth for 
$t\in [\tau-1, \tau]\subset (-\infty, 0)$.
Even more we assume that $\cm$ can be written on this time interval as a normal graph over
$M_\Sigma$, i.e.
$$M_t = \text{graph}_{\Sigma_t}(u(\cdot,t))$$
where $u(\cdot,t)$ is a smooth section of the normal bundle
$T^\perp\Sigma_t$, with sufficiently small $C^1$-norm. 
Note that any $n$-surface which is sufficiently close in $C^1$
to $\Sigma_t$ can written as such a normal graph. Since the evolution
of $\Sigma$ is self-similarly shrinking we can also write
$$M_t = \sqrt{-t}\cdot \text{graph}_{\Sigma}(v(\cdot,t)),$$
which implies that $u(p,t) = \sqrt{-t}\cdot v(\tilde{p},t)$, where
$\tilde{p}\in \Sigma$ is given by $\sqrt{-t}\cdot\tilde{p} = p$. We have the
following extension lemma:
\begin{lemma}\label{lem:extension}
  Let $\beta >1$, $\tau<0$. For every $\sigma>0$ there 
  exists a $\delta > 0$, depending only on $\sigma, \beta, \Sigma$, such
  that if $\cm$ is a unit density Brakke flow with $
  \Theta_{(0,0)}(\cm) \geq \Theta_{0,0}(\cm_\Sigma)$, which is a smooth
  graph over $\cm_\Sigma$ for $t\in [\beta\tau, \tau]\subset (-\infty,
  0)$ such that
  \begin{equation}
    \label{eq:extlem1}
    \|v\|_{C^{2,\alpha}(\Sigma\times [\beta\tau,\tau])}\leq \sigma 
  \end{equation}
  and 
  \begin{equation}
    \label{eq:extlem2}
    \sup_{t\in[\beta\tau,\tau]}\|v(\cdot,t)\|_{L^2(\Sigma)} \leq \delta\ ,
  \end{equation}
  then $\cm$ is a smooth graph over $\cm_\Sigma$ on $[\beta\tau, \tau/\beta]$ for an
  extension $\tilde{v}$ of $v$ with
  \begin{equation}
    \label{eq:extlem3}
    \|\tilde{v}\|_{C^{2,\alpha}(\Sigma\times
      [\beta\tau,\tau/\beta])}\leq \sigma\ .
  \end{equation}
\end{lemma}
\begin{proof}
  By changing scale, we can assume that $\tau = -1$. If the statement
  were false we could find a sequence of Brakke flows $\cm^k$ with $
  \Theta_{(0,0)}(\cm^k) \geq \Theta_{0,0}(\cm_\Sigma)$, which are smooth
  graphs over $\cm_\Sigma$ for $t\in [-\beta,-1]$. Furthermore we can assume that
  \eqref{eq:extlem1} holds for all $k$ and
  \begin{equation}
    \label{eq:extlem4}
    \sup_{t\in[-\beta,-1]}\|v(\cdot,t)\|_{L^2(\Sigma)} \leq \frac{1}{k}\ ,
  \end{equation}
but $\cm^k$ is not a smooth graph over $\cm_\Sigma$ for $t\in
[-1,-1/\beta]$ satisfying \eqref{eq:extlem3}. 
By the compactness theorem for Brakke flows there exists a subsequence
$\cm^{k'}$ such that
$$\cm^{k'}\rightarrow \cm'\ ,$$
where $\cm'$ is again a Brakke flow. Since the Gaussian density is upper
semi-continuous we have 
$$\Theta_{(0,0)}(\cm') \geq \Theta_{0,0}(\cm_\Sigma)$$
and \eqref{eq:extlem3}, \eqref{eq:extlem4} imply that
$$\cm' = \cm_\Sigma\ \ \text{for}\ t\in [-\beta,-1]\ ,$$
which in turn forces
$$\Theta_{(0,0)}(\cm') = \Theta_{0,0}(\cm_\Sigma)\, .$$
But this implies that $\cm'$ is self similarly shrinking for $t\in
(-\beta,0)$ and coincides with $\cm_\Sigma$ for $t\in (-\beta,0)$ (with unit
density). By Lemma \ref{lem:locreg}
the convergence 
$$\cm^{k'} \rightarrow \cm_\Sigma$$
is smooth on any compact subset of $\R^{n+k}\times(-\beta,0)$, which gives the
desired contradiction.
\end{proof}
To study mean curvature flows close to the evolution of $\Sigma$ it is 
convenient to consider the rescaled flow,
 i.e. if $F:M^n\times (t_1,t_2) \rightarrow
\mathbb{R}^{n+k},\ t_1<t_2\leq 0$ is a smooth mean curvature flow, the
rescaled embeddings
$$ \tilde{F}(\cdot,\tau):= \frac{1}{\sqrt{-t}}F(\cdot,t)$$
with $\tau = -\log(-t)$ have normal speed
\begin{equation}
  \label{eq:mcfresc}
  \Big(\frac{\partial}{\partial \tau} \tilde{F}\Big)^\perp =
  \vec{H}+\frac{x^\perp}{2}\ .
\end{equation}
The monotonicity formula \eqref{eq:huiskenmon}, centered at $(0,0)$, in the rescaled setting reads
\begin{equation*}
  \frac{d}{d\tau}\int_{\tilde{M}_\tau}\rho(x)
 \, d\mathcal{H}^n = -
  \int_{\tilde{M}_\tau} \Big|\vec{H}+ \frac{x^\perp}{2}\Big|^2\rho(x)\, d\ch^n\ ,
\end{equation*} 
 with $\rho(x)= (4\pi)^{-n/2}\exp\big(-|x|^2/4\big)$. If the surfaces
 $\tilde{M}_\tau:= 1/\sqrt{-t}\cdot M_t$ can be written as normal
 graphs 
over $\Sigma$, i.e.
$$\tilde{M}_\tau = \text{graph}_{\Sigma}(v(\cdot,\tau))\, ,$$
equation \eqref{eq:mcfresc} implies
\begin{equation}
  \label{eq:mcfrescgraph}
   \Big(\frac{\partial}{\partial \tau} v \Big)^\perp =
  \vec{H}+\frac{x^\perp}{2}\, ,
\end{equation}
where, as before, ${}^\perp$ denotes the projection onto the normal
space of $\tilde{M}_\tau$.
\section{Convergence}
\label{sec:convergence}
To show that the flow $\tilde{M}_\tau$ stays close to $\Sigma$ we treat the
Gaussian integral of the monotonicity formula as an ``energy
functional'' for surfaces which can be written as normal graphs over
$\Sigma$. Let $v$ be a smooth section of the normal bundle
$T^\perp\Sigma$ with sufficiently small $C^1$-norm, i.e. such that
$M:=\text{graph}_\Sigma(v)$ is a smooth $n$-surface, which is close in
$C^1$ to $\Sigma$. The energy $\mathcal{E}$ is then given by
\begin{equation*}
  \mathcal{E}(v)= \int_M \rho(x)\, d\ch^n(x) = \int_\Sigma \rho(y+v(y))\,
  J(y,v, \nabla^\Sigma v)\, d\ch^n(y)\ ,
\end{equation*}
where by the area formula $J$ is a smooth function with analytic
dependence on its arguments.  Furthermore it is uniformly convex in
$\nabla^\Sigma v$, if the $C^1$-norm of $v$ is sufficiently small. 
The first variation of $\mathcal{E}$ at $v$ is given by
\begin{equation*}
  \begin{split}
  \frac{\partial}{\partial
    s}\bigg|_{s=0}\!\!\!\!\!\!\!\mathcal{E}(v+sf) 
  &=  - \int_M
  \Big\langle \vec{H}+x^\perp/2, f\Big\rangle_{\R^{n+k}} \rho(x)\,
  d\ch^n(x)\\
  &= - \int_{\Sigma}
  h\Big(\Pi\big(\vec{H}+x^\perp/2\big), f\Big)\,
  \rho(y+v(y))\,  J(y, v, \nabla^\Sigma v)\,
  d\ch^n(y)\ , 
  \end{split}
\end{equation*}
where $h$ is the induced metric on the normal bundle to $\Sigma$, and
$\Pi$ the projection to the normal space of $\Sigma$ at $y$. Thus the
$L^2$-gradient operator of $\mathcal{E}$ is given by 
\begin{equation*}
\text{grad}\, \mathcal{E}(v) =  - \Pi\Big(\vec{H}+\frac{(y+v(y))^\perp}{2}\Big)\, \rho(y+v(y))\,  J(y, v, \nabla^\Sigma v)
\end{equation*}
where  $\vec{H}$ is the mean curvature operator of
$\text{graph}_\Sigma(v)$ at $y+v(y)$.

We aim to apply the Simon-{\L}ojasiewicz inequality as proven in \cite{Simon83}.
One can easily check that all conditions to apply Theorem 3 therein are
met, and we thus get that there are constants $\sigma_0>0$ and $\theta
\in (0,\tfrac{1}{2})$ such that if $v$ is a $C^{2,\alpha}$ section of
$T^\perp\Sigma$ with $|v|_{2,\alpha}<\sigma_0$ then
\begin{equation}
  \label{eq:conv4}
  \|\text{grad}\,\mathcal{E}(v)\|_{L^2(\Sigma)}\geq |\mathcal{E}(v)- 
  \mathcal{E}(0)|^{1-\theta}\, .
\end{equation}

We can assume that  $\sigma_0>0$ is small enough such that for any
normal section $v$ with $|v|_{C^{2,\alpha}}<\sigma_0$ the normal graph,
$\text{graph}_{\Sigma}(v)$, is a smooth $n$-surface. This estimate 
allows us to show that solutions of \eqref{eq:mcfrescgraph} stay 
close to $\Sigma$, provided they are bounded in $C^{2,\alpha}$.

\begin{lemma}
  \label{lem:L2control}
  Let $v: \Sigma\times[\tau_1,\tau_2]\rightarrow T^\perp\Sigma$ be a
  smooth solution of \eqref{eq:mcfrescgraph} with 
  $|v(\cdot, \tau)|_{2,\alpha}<\sigma_0$  for all
  $\tau\in[\tau_1,\tau_2]$ and $\mathcal{E}(v(\tau_2))\geq
  \Theta_{(0,0)}(\cm_\Sigma)$. Then there exists $\gamma > 0$,
  depending only on $\sigma_0, \mathcal{E}$ and $\Sigma$, such that
  \begin{equation}
    \label{eq:conv5}
    \int_{\tau_1}^{\tau_2} \big\|\tfrac{\partial v}{\partial\tau}
    \big\|_{L^2(\Sigma)}\, d\tau \leq \frac{1}{\gamma\theta}
    \big(\mathcal{E}(v(\tau_1))- \mathcal{E}(0))\big)^\theta
  \end{equation}
  and thus
  \begin{equation}
    \label{eq:conv6}
    \sup_{\tau\in[\tau_1,\tau_2]}\|v(\tau)-v(\tau_1)\|_{L^2(\Sigma)} 
    \leq \frac{1}{\gamma\theta}
    \big(\mathcal{E}(v(\tau_1))- \mathcal{E}(0)\big)^\theta\ .
  \end{equation}
\end{lemma}
\begin{proof}
  We have
  \begin{equation*}
    \begin{split}
      \frac{d}{d\tau}\mathcal{E}(v(\tau)) &= - 
      \int_{M_\tau}\Big|\vec{H}+ \frac{x^\perp}{2}\Big|^2\rho(x)\,
      d\ch^n\\
      &= - \Big( \int_{M_\tau}\Big|\vec{H}+ \frac{x^\perp}{2}\Big|^2\rho(x)\,
      d\ch^n\Big)^{1/2}\cdot\Big(\int_{\Sigma} \Big|\vec{H}+ 
      \frac{x^\perp}{2}\Big|^2\, \rho(x)\,  J(y)\,
      d\ch^n(y)\Big)^{1/2}\\
      &\leq - \Big( \int_{\Sigma}\Big|\text{grad}\,
      \mathcal{E}(v(\tau))|^2 \big(\rho(y+v(y))J(y)\big)^{-1}\,
      d\ch^n(y)\Big)^{1/2}\\
      &\quad\cdot\Big(\int_{\Sigma}
      \Big|\Big(\frac{\partial}{\partial \tau}v\Big)^\perp\Big|^2\, 
      \rho(y+v(y))\,  J(y)\,
      d\ch^n(y)\Big)^{1/2}\\
      &\leq -\, \gamma\, 
      \|\text{grad}\,\mathcal{E}(v)\|_{L^2(\Sigma)}\cdot 
      \big\|\tfrac{\partial v}{\partial\tau}\big\|_{L^2(\Sigma)}\ ,
    \end{split}
  \end{equation*}
where $\gamma>0$ depends only on $\sigma_0$ and $\Sigma$. Since
$\mathcal{E}(v(\tau))$ is a decreasing function of $\tau$ 
and $\mathcal{E}(0)=\Theta_{(0,0)}(\cm_\Sigma)$ this yields
together with \eqref{eq:conv4},
\begin{equation*}
  \begin{split}
    -\frac{d}{d\tau}\big(\mathcal{E}(v(\tau))-\mathcal{E}(0)\big)^\theta
    &\geq \gamma\, \theta\,
    \big(\mathcal{E}(v(\tau))-\mathcal{E}(0)\big)^{\theta-1}
    \|\text{grad}\,\mathcal{E}(v)\|_{L^2(\Sigma)}\cdot 
      \big\|\tfrac{\partial v}{\partial\tau}\big\|_{L^2(\Sigma)}\\
      &\geq \gamma\, \theta\, \big\|\tfrac{\partial v}{\partial\tau}
      \big\|_{L^2(\Sigma)}\ .
  \end{split}
\end{equation*}
Integrating this inequality yields \eqref{eq:conv5}, which implies 
\eqref{eq:conv6}.
\end{proof}
\begin{proof}[Proof of Theorem \ref{thm:mainthm}]
Let
$\beta=(t_1+t_2)/(2t_2)$ and $\delta>0$ be given by lemma
\ref{lem:extension}, for $\sigma = \sigma_0$. Choose
$0<\tilde{\sigma}<\sigma_0$ such
that $|v|_{C^{2,\alpha}}<\tilde{\sigma}$ implies
\begin{itemize}
\item[i)] $\|v\|_{L^2(\Sigma)}\leq \delta/3$,\\[-1ex]
\item[ii)]
  $\frac{1}{\gamma\theta}\big|\mathcal{E}(v)-\mathcal{E}(0)\big|^\theta\leq
\frac{2\delta}{3}\ ,$
\end{itemize}
where $\gamma$ is the constant given by Lemma \ref{lem:L2control}.

By Lemma \ref{lem:locreg} we can assume that $\cm$ is actually a smooth
graph over $\Sigma_t$ for $t \in ((t_1+t_2)/2, t_2)$,
such that $M_t = \text{graph}_{\Sigma_t}(u(\cdot, t))$. 
Even more, we can assume that in the rescaled setting for
$$\tilde{M}_\tau=\frac{1}{\sqrt{-t}}M_t = \text{graph}_\Sigma(v(\cdot,\tau)) $$
 the normal section $v(\tau)$ satisfies 
$|v(\tau)|_{C^{2,\alpha}}< \tilde{\sigma}$, where $\tau =
-\log(-t)$ and $\tau \in \big(\tau_0-\log(\beta),\tau_0\big)$,
for $\tau_0=-\log(-t_2)$. Lemma \ref{lem:extension} implies that
$\tilde{M}_\tau$ is actually graphical over $\Sigma$ for an extension
$\tilde{v}$ of $v$ on $\big(\tau_0-\log(\beta),
\tau_0+\log(\beta)\big)$ satisfying $|\tilde{v}(\tau)|_{C^{2,\alpha}}<
\sigma_0$. But note that Lemma \ref{lem:L2control} then implies that for
$\tau \in [\tau_0, \tau_0+\log(\beta))$
\begin{equation*}
  \|\tilde{v}(\tau)\|_{L^2(\Sigma)}\leq \|v(\tau_0)\|_{L^2(\Sigma)}+\frac{1}{\gamma\theta}
    \big(\mathcal{E}(\tilde{v}(\tau_0))-
    \mathcal{E}(0))\big)^\theta
    \leq \frac{\delta}{3}+ \frac{2\delta}{3} = \delta\ .
\end{equation*}
This now allows
us to iterate the previous step to find that actually $\tilde{M}_\tau$ is a
smooth graph over $\Sigma$ for an extension $\tilde{v}(\tau)$ of $v$
for all $\tau \in (\tau_0,\tau_0+2\log(\beta))$ such that 
$|\tilde{v}(\tau)|_{C^{2,\alpha}}<\sigma_0$. Again we have
\begin{equation*}
  \begin{split}
  \|\tilde{v}(\tau)\|_{L^2(\Sigma)}&\leq \|v(\tau_0)\|_{L^2(\Sigma)}+\frac{1}{\gamma\theta}
    \big(\mathcal{E}(\tilde{v}(\tau_0+\log(\beta)))-
    \mathcal{E}(0))\big)^\theta\\
    &\leq \frac{\delta}{3}+\frac{1}{\gamma\theta}
    \big(\mathcal{E}(v(\tau_0))-
    \mathcal{E}(0))\big)^\theta \leq \delta\ 
    \end{split}
\end{equation*}
for $\tau \in (\tau_0,\tau_0+2\log(\beta))$, since
$\mathcal{E}(v(\tau))$ is decreasing in $\tau$ and
$$ \mathcal{E}(v(\tau)) \geq \Theta_{(0,0)}(\cm)\geq
\Theta_{(0,0)}(\cm_\Sigma)=\mathcal{E}(0)\ .$$
This can be iterated
to yield that there is an extension $\tilde{v}(\tau)$ of $v$
for all $\tau \in (\tau_0,\infty)$ such that 
$|\tilde{v}(\tau)|_{C^{2,\alpha}}<\sigma_0$ for all $\tau\geq
\tau_0$. Standard estimates for graphical solutions imply the bounds on
all higher derivatives.

The estimate \eqref{eq:conv5} implies that
$$\int_{\tau_0}^{\infty} \big\|\tfrac{\partial v}{\partial\tau}
    \big\|_{L^2(\Sigma)}\, d\tau \leq \frac{1}{\gamma\theta}
    \big(\mathcal{E}(v(\tau_0))- \mathcal{E}(0))\big)^\theta\leq
    \frac{\delta}{3}\ .$$
By the bounds on all higher derivatives this gives that
$$\tfrac{\partial v}{\partial\tau} \rightarrow 0 \ \ \text{in}\ C^k$$
for all $k$ uniformly as $\tau \rightarrow \infty$. Thus there is a
sequence of times $\tau_i\rightarrow \infty$ such that
$v(\cdot,\tau_i+\tau)\big|_{\tau \in (0,1)}$ converges uniformly
in $C^k(\Sigma \times (0,1))$, for all $k$, 
to a time independent solution $v'$ 
of \eqref{eq:mcfrescgraph}. In other words
$$ \Sigma' := \text{graph}_\Sigma(v') $$
gives rise to a self-similarly shrinking solution of MCF. Furthermore
note that 
$$\Theta_{(0,0)}(\cm) = \Theta_{(0,0)}(\cm_{\Sigma'}) =
\mathcal{E}_{\Sigma'}(0)\ .$$
Thus we can repeat the whole argument before, writing now $M_{\tau}$
as normal graphs over $\Sigma'$, given by a time-dependent normal section $w(\tau)$. Since now
$\mathcal{E}_{\Sigma'}(w(\tau))\rightarrow \mathcal{E}_{\Sigma'}(0)$
and $w(\tau_i)\rightarrow 0$, 
the estimate \eqref{eq:conv6} already implies that 
$$ w(\tau) \rightarrow 0 \ \ \text{uniformly as}\ \tau\rightarrow
\infty\ .$$
To get the claimed decay rate we follow ideas in
\cite{LSimon96}. Note that from a calculation as in the
proof of Lemma \ref{lem:L2control} and \eqref{eq:conv4}, we get for
some $\tilde{\gamma}>0$
\begin{equation*}
  \begin{split}
    \frac{d}{d\tau}\big(\mathcal{E}_{\Sigma'}(w(\tau))-\mathcal{E}_{\Sigma'}(0)\big)
    &\leq
    -\tilde{\gamma}\,\|\text{grad}\,\mathcal{E}_{\Sigma'}\|^2_{L^2(\Sigma')}\\
    &\leq -\tilde{\gamma}\, \big(\mathcal{E}_{\Sigma'}(w(\tau))
    -\mathcal{E}_{\Sigma'}(0)\big)^{2(1-\theta)}\, .
  \end{split}
\end{equation*}
Integrating this inequality yields
$$\big|\mathcal{E}_{\Sigma'}(w(\tau))
    -\mathcal{E}_{\Sigma'}(0)\big| \leq C\frac{1}{\tau^{1+\alpha}}
    ,$$
where $\alpha = 2\theta/(1-2\theta)$. The estimate  \eqref{eq:conv4} gives
$$\|v(\tau)\|_{L^2(\Sigma')} \leq C\frac{1}{\tau^{\theta(1+\alpha)}}  .$$
Rewriting $\tau = -\log(-t)$ yields the claimed decay for
the $L^2$-norm. The higher norms follow by interpolation.
\end{proof}
\begin{proof}[Proof of Corollary \ref{thm:maincor}] We can assume
  w.l.o.g. that $(x_0,t_0)=(0,0)$. Since $\cm_\Sigma$
  is smooth with unit density, and $\Sigma$ is compact, Lemma
  \ref{lem:locreg}
  implies that the convergence
  $$\mathcal{D}_{\lambda_i}(\cm) \rightarrow \cm_\Sigma$$
  is smooth on $\R^{m+k}\times (t_1,t_2)$, for any $t_1<t_2<0$. Since
  $\Theta_{(0,0)}(\cm)=\Theta_{(0,0)}(\cm_\Sigma)$ we can apply
  Theorem \ref{thm:mainthm} to see that $\cm$ is a smooth graph over
  $\cm_\Sigma$ for $t\in (t_1,0)$. The assumption that $\cm_\Sigma$ is
  a tangent flow of $\cm$ at $(0,0)$
  implies that after rescaling, there is a sequence of times $\tau_i
  \rightarrow \infty$ such that $\tilde{M}_{\tau_i}\rightarrow \Sigma$. Thus
  $\Sigma'=\Sigma$ and $\tilde{M}_\tau \rightarrow \Sigma$ as $\tau \rightarrow \infty$.
  This implies the statement of the corollary.  
\end{proof}

\end{document}